\documentclass[12pt]{amsart}

\usepackage{fullpage}
\usepackage{dsfont}
\usepackage{url}
\usepackage{multirow}
\usepackage[all]{xy}
\usepackage{algorithm}
\usepackage[colorlinks]{hyperref}
\usepackage{algpseudocode}

\usepackage{amsmath, amssymb, amsfonts, amsthm}
\usepackage{tikz}

\usepackage{listings, graphicx}
\lstdefinelanguage{Sage}[]{Python}
{morekeywords={False,sage,True},sensitive=true}
\definecolor{dblackcolor}{rgb}{0.0,0.0,0.0}
\definecolor{dbluecolor}{rgb}{0.01,0.02,0.7}
\definecolor{dgreencolor}{rgb}{0.2,0.4,0.0}
\definecolor{dgraycolor}{rgb}{0.30,0.3,0.30}

\theoremstyle{plain}
\newtheorem{theorem}{Theorem}[section]

\newtheorem{proposition}[theorem]{Proposition}
\newtheorem{lemma}[theorem]{Lemma}
\newtheorem{example}[theorem]{Example}
\newtheorem{definition}[theorem]{Definition}
\theoremstyle{remark}
\newtheorem*{remark}{Remark}

\def\Gal{{\rm Gal}}

\def\Im{{\rm Im}}

\def\End{{\rm End}}
\def\Z{{\mathbb{Z}}}
\def\F{{\mathbb{F}}}
\def\C{{\mathbb{C}}}
\def\Q{{\mathbb{Q}}}
\def\O{{\mathcal{O}}}
\def\Jac{{\rm Jac}}

\def\Aut{{\rm Aut}}
\newcommand{\Sp}{\operatorname{Sp}}

\makeatletter
\def\smallunderbrace#1{\mathop{\vtop{\m@th\ialign{##\crcr
   $\hfil\displaystyle{#1}\hfil$\crcr
   \noalign{\kern3\p@\nointerlineskip}
   \tiny\upbracefill\crcr\noalign{\kern3\p@}}}}\limits}
   
\def\smalloverbrace#1{\mathop{\vbox{\m@th\ialign{##\crcr\noalign{\kern3\p@}
     \tiny\downbracefill\crcr\noalign{\kern3\p@\nointerlineskip}
      $\hfil\displaystyle{#1}\hfil$\crcr}}}\limits}

\subjclass[2010]{Primary 11G15; Secondary 11G30, 14H42, 14Q05}

\begin{document}

\title{Constructing genus 3 hyperelliptic Jacobians with CM}
\author{Jennifer S. Balakrishnan}
\address{Jennifer S. Balakrishnan, Mathematical Institute, University of Oxford, Woodstock Road, Oxford OX2 6GG, UK}
\email{balakrishnan@maths.ox.ac.uk}
\author{Sorina Ionica}
\address{Sorina Ionica, MIS, Universit\'e de Picardie Jules Verne, 33 Rue Saint Leu, 80000 Amiens, France}
\email{sorina.ionica@m4x.org}
\author{Kristin Lauter}
\address{Kristin Lauter, Microsoft Research, 1 Microsoft Way, Redmond, WA 98062, USA}
\email{klauter@microsoft.com}
\author{Christelle Vincent}
\address{Christelle Vincent, Department of Mathematics and Statistics, The University of Vermont, 16 Colchester Avenue, Burlington VT 05401, USA}
\email{christelle.vincent@uvm.edu}

\begin{abstract} Given a sextic CM field $K$, we give an explicit method for finding all genus $3$ hyperelliptic curves defined over $\C$ whose Jacobians are simple and have complex multiplication by the maximal order of this field, via an approximation of their Rosenhain invariants. Building on the work of Weng~\cite{WengGenus3}, we give an algorithm which works in complete generality, for any CM sextic field $K$, and computes minimal polynomials of the Rosenhain invariants for any period matrix of the Jacobian.  This algorithm can be used to generate genus 3 hyperelliptic curves over a finite field $\F_p$ with a given zeta function by finding roots of the Rosenhain minimal polynomials modulo $p$.
\end{abstract}

\maketitle

\section{Introduction}
We consider the problem of constructing genus $3$ hyperelliptic curves defined over $\C$ with the property that their Jacobians are simple and admit complex multiplication (CM) by the maximal order of a sextic field. The interest in this question stems from the situation in genera $1$ and $2$, where curves over a finite field with CM by a given field can be found by first computing curves defined over number fields with CM by the field of interest and then reducing these curves modulo a prime ideal, under some hypotheses that guarantee that the prime splits and that the endomorphism ring does not become larger.

In genus $3$, however, the situation is more interesting. Up to isomorphism over $\C$, every simple principally polarized abelian variety (ppav) of dimension $3$ is the Jacobian of a complete smooth projective curve of genus $3$. Furthermore, if $X$ is such a curve, then by Riemann-Roch, $X$ is  isomorphic either to a hyperelliptic or a plane quartic curve. If $A$ is a simple ppav of dimension $3$ that is isomorphic to the Jacobian of a hyperelliptic curve, resp. a plane quartic curve, we will call it a \emph{hyperelliptic Jacobian}, resp. a \emph{plane quartic Jacobian}. We note that the subspace of hyperelliptic Jacobians has codimension 1 in the moduli space of ppav of dimension 3.

If we are interested only in generating hyperelliptic curves whose Jacobians have CM,  then given a sextic CM field $K$, we consider the set of simple ppav having CM by the maximal order $\O_K$ and ask some natural questions: Does this set contain hyperelliptic Jacobians? And, what conditions on $K$ determine if this set contains hyperelliptic Jacobians? 

Our work takes steps toward answering these questions by presenting an algorithm that, given a CM sextic field $K$, first constructs a period matrix for each isomorphism class of simple ppav with CM by $\O_K$, then verifies computationally if the abelian variety is the Jacobian of a hyperelliptic curve. If this is the case, it computes minimal polynomials for the Rosenhain invariants of the hyperelliptic curve. After the computation is completed, we check experimentally that the hyperelliptic curves we computed have CM, as explained in $\S$\ref{results}.

Using our implementation of the algorithm, we have carried out the computations described above for all Galois CM sextic fields $K$ with class number $1$. Some examples of the minimal polynomials computed can be found in $\S$\ref{results}, along with a list of such fields $K$ that admit a hyperelliptic Jacobian. The code we have written is available on GitHub \cite{github}. We conjecture that this list is complete: there are exactly four hyperelliptic curves, up to isomorphism, with CM by a Galois sextic field of class number $1$.

In 2001, Weng \cite{WengGenus3} carried out computations similar to those presented here. She showed that if $A$ is a simple ppav of dimension $3$ having CM by the maximal order of a sextic field $K$ such that $\Q(i)\subset K$, then $A$ is the Jacobian of a hyperelliptic curve. While she restricts herself to such fields, we do not, as stated above. In other ways, however, she computes more than we do: Weng approximates the Rosenhain coefficients and uses them to get the Shioda invariants of the hyperelliptic curves. In addition, she provides models for several hyperelliptic curves of genus $3$ with CM by sextic fields $K$. In particular, she exhibits four hyperelliptic curves with CM whose model is defined over $\Q$. Unfortunately, Weng's implementation is not publicly available, and the paper contains typos and imprecisions,  which make it is hard to directly reproduce her computations. In this work, we correct these imprecisions and give a proof of the correctness of the algorithm. See Remark~\ref{compareWeng} and $\S$\ref{list} for further comments and comparisons with Weng's work.

We note that so far all sextic CM fields admitting a hyperelliptic Jacobian that have been found contain either a fourth or a seventh root of unity. It is of interest to attempt to find a sextic CM field $K$ admitting a hyperelliptic Jacobian but containing only the roots of unity $-1$ and $1$, with the aim of determining experimentally that this is not a necessary condition for $K$ to admit a hyperelliptic Jacobian.

For this reason and to answer the questions presented earlier, in forthcoming work we plan to use this algorithm to explore the case of general CM sextic $K$. Currently the issues that prevent us from carrying out these computations have to do with precision: we need to estimate the tail of the infinite series giving the theta constants so that we can ensure \emph{a priori} that our results are correct to a certain precision. To obtain these estimates, we need an algorithm that takes any period matrix $Z \in \mathcal{H}_3$ and computes a representative in the same $\Sp_6(\Z)$-equivalence class but belonging to a suitable fundamental domain. We give more details in $\S$\ref{results}.

We particularly aim at finding examples of fields $K$ admitting both a hyperelliptic and a plane quartic Jacobian, if they exist. This question is closely related to the construction and use in cryptography of genus 3 hyperelliptic Jacobians defined over finite fields, with CM by a given sextic field. Indeed, it is well known that discrete log attacks on plane quartic Jacobians are more efficient than on genus 3 hyperelliptic Jacobians~\cite{Diem,Gaudry,Laine}. Consequently, evaluating the security of a genus $3$ hyperelliptic Jacobian requires a good understanding of the types of Jacobians appearing in its isogeny class~\cite{Smith}.

In the body of the paper we present some background and references for our algorithm, as well as some results which were needed to carry out the computation. In particular, in Example \ref{etabar} we present an example of a period matrix in a $\Gamma_2$-equivalence class not previously considered by Mumford \cite{Mumford1}, but which we must consider to make our algorithm truly applicable to any period matrix. In $\S$\ref{vanishing} we also verify, using Mumford's \cite{Mumford1,Mumford} and Poor's work \cite{poor}, that the Thomae formulae can be used to compute hyperelliptic models starting from any period matrix, without making a specific choice for the basis for homology. Although similar computations have been performed before~\cite{weber,WengGenus3}, we were not able to find a proof of these formulae in the literature.

This paper is organized as follows. In $\S$\ref{preliminaries} we present certain results needed to generate all primitive CM types of a CM sextic field, up to equivalence. These CM types are needed to carry out the algorithm given by Koike and Weng \cite{koikeweng} to enumerate period matrices of simple ppav with CM by a fixed field $K$, which is presented in Appendix \ref{periodmatrixalgorithm} for completeness. In $\S$\ref{mapseta}, we introduce a set of maps denoted $\Xi_g$, first defined by Poor \cite{poor}, and show how to attach such a map to a hyperelliptic Jacobian. In $\S$\ref{vanishing}, we introduce theta functions, the Vanishing Criterion and a formula to compute the model of a hyperelliptic curve given the period matrix of its Jacobian. In $\S$\ref{computingeta} we give, and provide justification for, our algorithm to compute the map $\eta$ attached to a hyperelliptic Jacobian. Finally in $\S$\ref{results} we present our complete algorithm, as well as selected examples of the minimal polynomials for the Rosenhain invariants that we have computed.

\section{Computing period matrices}\label{preliminaries}
Given a CM field $K$, to compute ppav with CM by the full ring of integers of $K$, we rely on the CM theory of Shimura and Taniyama~\cite{Shimura}. In genus 3, an explicit construction was presented by Koike and Weng~\cite{koikeweng}. We use their algorithm and present here a single result needed to complete our work, as well as the most basic facts needed to put this result in context. For further details, we refer the reader to Lang \cite{lang} or Birkenhake and Lange \cite{Birkenhake}.

We will use the term \emph{period matrix} to refer to an element $Z \in \mathcal{H}_g$, where
\begin{equation}\label{eq:upperhalf}
\mathcal{H}_g = \{ M \in M_{g \times g}(\mathbb{C}) : M^{T} = M, \Im(M) >0 \}.
\end{equation}
To such a period matrix we can associate a lattice $L_Z$ generated by the columns of the matrix $(\mathds{1}_g,Z)$, where $\mathds{1}_g$ is the $g \times g$ identity matrix.

This lattice gives rise to an abelian variety $A$ whose underlying torus is isomorphic to $\C^{g}/L_Z$. In this paper, we focus on the case where $\End(A)= \O_K$, for $\O_K$ the ring of integers of a CM field $K$ of degree $2g$. We will say that such an $A$ has \emph{CM by $K$}, and by this we will always mean that the endomorphism ring of $A$ is the full ring of integers $\O_K$.

To each abelian variety of dimension $g$ defined over $\C$ with CM by $K$ is attached a $g$-tuple of complex embeddings of $K$ (i.e., embeddings of $K$ into $\mathbb{C}$), no two of which are complex conjugates, called the variety's \emph{CM type}. Conversely, when constructing such an abelian variety, we must first choose a CM type.
An abelian variety over $\C$ with CM by the ring of integers $\mathcal{O}_K\subset K$ is given by $A=\C^g/\Phi(\mathfrak{a})$, where $\mathfrak{a}$ is an ideal of $\mathcal{O}_K$ and $\Phi$ is a CM type. This variety is said to be of CM type $(K,\Phi)$.

We are interested only in constructing \emph{simple} abelian varieties, a property which is completely controlled by the choice of CM type. Indeed, fix $(K, \Phi)$ a CM type and let $L$ be the Galois closure of $K$ over $\Q$. Throughout, let $G$ be the Galois group $\Gal(L/\Q)$, and set $H= \Gal(L/K)$. Define the sets
\begin{equation*}
S  = \{ \sigma \in G : \sigma|_K = \phi_i, \text{ for one of } i = 1,\ldots, g\}, \quad \text{and} \quad H' =\{ \gamma \in G : S \gamma = S   \}.
\end{equation*}
A CM type $(K, \Phi)$ is called \emph{primitive} (or \emph{simple} in Lang \cite{lang}) if $H = H'$. It can be shown that an abelian variety of CM type $(K, \Phi)$ is simple if and only if its CM type is primitive. Since we are interested in constructing complex ppav that are simple, we will restrict our attention to primitive CM types.  Two CM types $\Phi_1$ and $\Phi_2$ are said to be \emph{equivalent} if there is an automorphism $\sigma$ of $K$ such that $\Phi_1 = \Phi_2 \sigma$. We have the following result:
\begin{proposition}[Streng {\cite[Lemmata I.5.4 and I.5.6]{streng}}]
Let $A_1$ and $A_2$ be abelian varieties over $\C$ with CM types $\Phi_1$ and $\Phi_2$, where $\Phi_1$ and $\Phi_2$ are primitive CM types for a common CM field $K$. If $A_1$ and $A_2$ are isomorphic, then the two CM types are equivalent.
Moreover, if two CM types $\Phi_1$ and $\Phi_2$ are equivalent, then the set of isomorphism classes of ppav with CM type $\Phi_1$ coincides with the set of isomorphism classes of ppav with CM type $\Phi_2$.
\end{proposition}
Since we are interested in enumerating abelian varieties with CM by a certain field up to isomorphism, it suffices to consider only one CM type from each equivalence class of equivalent CM types. In our case of interest, $g=3$ and $K$ is a sextic CM field. There are thus four possible isomorphism classes of Galois groups $G$ of the Galois closure $L$ of $K$ over $\Q$: $\Z/6\Z$, $\Z/2\Z\times S_3$, $(\Z/2\Z)^3\rtimes \Z/3\Z$ or $(\Z/2\Z)^3\rtimes S_3$. (Note that Weng \cite{WengGenus3} has a typo in the orientation of the symbol $\rtimes$.) Weng determines primitive CM types for each isomorphism class of $G$, and with this it is straightforward to determine the equivalence classes of equivalent CM types:
\begin{proposition}
	Let $K$ be a CM sextic field.
	\begin{enumerate}
		\item \label{galoiscase} If $G \cong \Z/6\Z$, $K$ has six primitive CM types, and they are all equivalent.
		\item If  $G \cong \Z/2\Z\times S_3$, $K$ has six primitive CM types, and three equivalence classes of equivalent primitive CM types.
		\item If $G \cong (\Z/2\Z)^3\rtimes \Z/3\Z$ or $G \cong (\Z/2\Z)^3\rtimes S_3$, $K$ has eight primitive CM types, and four equivalence classes of equivalent primitive CM types.
	\end{enumerate}
\end{proposition}
\begin{proof}
	In any case, $K$ has $2^3 =8$ CM types. For each case, the number and characterization of primitive CM types follow from Weng's work \cite[Lemma 3.1]{WengGenus3}.
	In part \ref{galoiscase}, the fact that all primitive CM types are equivalent follows again from Weng \cite[Theorem 3.5]{WengGenus3}.
	For the other parts, we use the fact that $\Aut(K)$ contains only the identity and complex multiplication. Therefore, a primitive CM type is only equivalent to its complex conjugate.
\end{proof}

With these results giving us a complete list of equivalence classes of equivalent primitive CM types for a given field $K$, we can apply Koike and Weng's \cite{koikeweng} algorithm to enumerate data for all isomorphism classes of simple ppav with CM by $K$. We briefly recall this method in Algorithm~\ref{alg:periodmatrix}, presented in Appendix \ref{periodmatrixalgorithm}, and refer the reader to \cite{koikeweng} for full details.

\section{The map $\eta$ attached to a hyperelliptic period matrix}\label{mapseta}
Given a hyperelliptic period matrix $Z$, Mumford \cite{Mumford1} constructs a certain map $\eta$. This map is crucial to the understanding of hyperelliptic Jacobians: First, its values control the vanishing of certain theta functions in such a way that the hyperelliptic Jacobians can be characterized by this vanishing property. Secondly, knowledge of a map $\eta$ attached to a period matrix allows one to recover a model for the hyperelliptic curve. These two phenomena are explained in $\S$\ref{vanishing}. Here we begin by describing the set of such maps $\eta$ that arise from hyperelliptic Jacobians and showing how to construct these maps given a hyperelliptic period matrix.

Throughout this section, we will take the convention that if $x \in \C^{2g}$, then $x = (x_1, x_2)$, with $x_i \in \C^g$; in other words $x_1$ will denote the vector of the first $g$ entries of $x$, and $x_2$ will denote the vector of the last $g$ entries of $x$. For a vector $x$, we will write $x^T$ for the transpose, and whenever matrix multiplication is involved, $x$ is taken to be a column vector.

\subsection{Eta maps}
Throughout, we let $B = \{1, 2, \ldots, 2g+1, \infty\}$. For any two subsets $S_1, S_2 \subseteq B$, we define
\begin{equation*}
S_1 \circ S_2 = (S_1 \cup S_2) - (S_1 \cap S_2),
\end{equation*}
the symmetric difference of the two sets. For $S \subseteq B$ we also define $S^c = B - S$, the complement of $S$ in $B$. Then we have that the set
\begin{equation*}
\{S \subseteq B : \# S \equiv 0 \pmod{2} \} / \{S \sim S^c\}
\end{equation*}
is a commutative group under the operation $\circ$, of order $2^{2g}$, with identity $\emptyset \sim B$. Since $S \circ S = \emptyset$ for all $S \subseteq B$, this is a group of exponent $2$. Therefore this group, which we denote $G_B$, is isomorphic to $(\Z/2\Z)^{2g}$.

We also need some functions on elements of $(1/2)\Z^{2g}$. Given $\xi \in (1/2)\Z^{2g}$, we continue to write $\xi = (\xi_1, \xi_2)$, as explained at the beginning of this section.

\begin{definition}\label{def:estar}
	For $\xi \in (1/2)\Z^{2g}$, let $e_*(\xi) = \exp(4\pi i \xi_1^T  \xi_2).$
\end{definition}

\begin{definition}
	For $\xi, \zeta \in (1/2)\Z^{2g}$, let $e_2(\xi,\zeta) = \exp(4\pi i \xi^T J\zeta)$, with $J = \left(\begin{smallmatrix} 0 & \mathds{1}_g \\ - \mathds{1}_g & 0 \end{smallmatrix}\right)$.
\end{definition}

We note that these two functions are related by the following formula: For $\xi_i \in (1/2)\Z^{2g}$,
\begin{equation*}
e_*\left(\sum_{i=1}^k \xi_i\right) = \prod_{i<j} e_2(\xi_i,\xi_j) \prod_{i=1}^k e_*(\xi_i).
\end{equation*}

We are now ready to define the set of maps of interest.

\begin{definition}\label{def:Xi}
	Following Poor \cite{poor}, we define the set $\Xi_g$ to contain the maps $\eta \colon P(B) \to (1/2)\mathbb{Z}^{2g}$, where $P(B)$ is the power set of $B$, satisfying the following properties:
	\begin{enumerate}
		\item $\eta(\{\infty\}) = 0$.
		\item \label{property:sum} For all $S \subseteq B$, $\eta(S) = \sum_{i \in S} \eta(\{i\})$.
		\item For all $S \subseteq B$, $\eta(S) = \eta(S^c)$ and the induced map $\eta \colon G_B \cong (1/2) \mathbb{Z}^{2g}/\mathbb{Z}^{2g}$ is a group isomorphism.
		\item For all sets $S_1$ and $S_2$ such that $\# S_1$, $\# S_2 \equiv 0 \pmod{2}$, $e_2(\eta(S_1),\eta(S_2)) = (-1)^{\#(S_1 \cap S_2)}$.
		\item \label{property:parity} There is $U_{\eta} \subset B$ such that $\# U_{\eta} \equiv g+1 \pmod{4}$ and for all $S$ such that $\# S \equiv 0 \pmod{2}$, we have $e_*(\eta(S)) = (-1)^{(g+1-\#(S \circ U_{\eta}))/2}$.
	\end{enumerate}
\end{definition}

\begin{remark}As noted in the proof of Lemma 1.4.13 of \cite{poor}, the set $U_{\eta}$ is none other than $\{i \in B - \{\infty\} : e_*(\eta(\{i\})) = -1 \} \cup \{\infty \}$.
\end{remark}

For ease of notation, for any map $\eta \in \Xi_g$, we will henceforth denote
\begin{equation*}\eta_\infty = \eta(\{\infty\})=0, \eta_1 = \eta(\{1\}), \ldots, \eta_{2g+1} = \eta(\{2g+1\}),\end{equation*}
and for any set $S \subset B$, we write $\eta_S = \eta(S).$

\begin{definition} Two maps $\eta$ and $\theta$ in $\Xi_g$ are said to be in the same \emph{class} if they are equal as maps to $(1/2) \mathbb{Z}^{2g}/\mathbb{Z}^{2g}$.
\end{definition}

Because of property \ref{property:sum} in Definition \ref{def:Xi}, any map $\eta \in \Xi_g$ is determined by its values $\eta_1, \ldots \eta_{2g+1}$. We have the following converse:

\begin{proposition}[Poor {\cite[Lemma 1.4.13]{poor}}]\label{prop:asygetic}
	Any ordered tuple $(\alpha_i)$ of $2g+1$ vectors in $(1/2)\mathbb{Z}^{2g}/\Z^{2g}$ gives rise to a class of maps $\eta \in \Xi_g$ via $\eta_i \equiv \alpha_i \pmod{\Z^{2g}}$ and $\eta_S = \sum_{i \in S} \eta_i$ if it satisfies the following conditions:
	\begin{enumerate}
		\item The $\alpha_i$'s span $(1/2)\mathbb{Z}^{2g}/\Z^{2g}$ as an $\mathbb{F}_2$-vector space.
		\item $\sum_{i=1}^{2g+1} \alpha_i = 0$.
		\item \label{property:e2} $e_2(\alpha_i,\alpha_j) = -1$ for each pair $i \neq j$.
	\end{enumerate}
	In fact, there is a bijection between the set of such tuples $(\alpha_i)$ and the classes of maps in $\Xi_g$.
\end{proposition}

We note that an ordered tuple satisfying these three conditions is commonly called an \emph{asygetic basis} in the literature. We will avoid this technical term and simply speak of the values $\eta_1, \ldots \eta_{2g+1}$ of a given  class of maps $\eta$, where $\eta$ is understood to be any representative of the class, and the entries of the values $\eta_i$ have all been reduced modulo $\Z$.

\subsection{Associating a map $\eta$ to a hyperelliptic period matrix}

In this section, let $X$ be a smooth complete hyperelliptic curve of genus $g$ defined over $\C$. As explained in the literature, for example \cite{Birkenhake}, a choice of period matrix $Z \in \mathcal{H}_g$ is equivalent to a choice of symplectic basis, $A_i$, $B_i$, for the homology group $H_1(X,\Z)$ of the curve. Indeed, without any further choice, there exists a unique basis $\omega_i$ of holomorphic differentials on $X$ such that
\begin{equation*}
\int_{A_i} \omega_i = 1 \quad \text{and} \quad \int_{A_i} \omega_j = 0, \, i \neq j.
\end{equation*}
Then $Z$ is the matrix given by $\int_{B_i} \omega_j$. Conversely, any period matrix is obtained in this manner.

Still without any further choices, we can obtain an Abel-Jacobi map
\begin{equation*}
AJ \colon \Jac(X)  \to \C^g/L_Z, \qquad
\sum_{k=1}^s P_k - \sum_{k=1}^s Q_k  \mapsto \left(\sum_{k=1}^s \int_{Q_k}^{P_k} \omega_i\right)_i,
\end{equation*}
which is well-defined since the value of each path integral on $X$ is well-defined up to the value of integrating the differentials $\omega_i$ along the basis elements $A_i$, $B_i$, and thus  up to elements of $L_Z$.

We can further choose to label the $2g+2$ branch points of the hyperelliptic map $\pi \colon X \to \mathbb{P}^1$, $P_1, P_2, \ldots, P_{2g+1}, P_{\infty}$. Given this second choice, we can give a group isomorphism (see \cite[Corollary 2.11]{Mumford} for details) between the $2$-torsion of the Jacobian of $X$ and the group $G_B$ in the following manner: To each set $S \subseteq B$ such that $\# S \equiv 0 \pmod{2}$, associate the divisor class of the divisor
\begin{equation}\label{eq:2torsion}
e_S = \sum_{i \in S} P_i - (\#S) P_{\infty}.
\end{equation}

In turn, this isomorphism gives rise to a class of maps $\eta \in \Xi_g$ by sending $S \subseteq B$ to the unique vector $\eta_S$ in $(1/2)\mathbb{Z}^{2g}/\Z^{2g}$ such that $AJ(e_S)=(\eta_S)_2 + Z (\eta_S)_1$. (The fact that $\eta \in \Xi_g$ is shown in \cite[Proposition 1.4.9]{poor}.) Since there are $(2g+2)!$ different ways to label the $2g+2$ branch points of a hyperelliptic curve $X$ of genus $g$, there are several ways to assign a class in $\Xi_g$ to a matrix $Z \in \mathcal{H}_g$.

We give here a diagram to illustrate the maps described above:
\begin{center}
\begin{tikzpicture}[
            point/.style={circle, fill=white, minimum size=0.1pt, inner sep=0pt},
            mu1/.style={black},
        ]
        \node[point] (A00) at (-1,1) {$G_B$};
        \node[point] (A01) at (5,1) {$\Jac(X)[2](\C)$};
        \node[point] (A02) at (10,1) {$(1/2)\mathbb{Z}^{2g}/\Z^{2g}$};
        \node[point] (A10) at (-1,0) {$S$};
        \node[point] (A11) at (5,0) {$e_S$};
        \node[point] (A12) at (10,0) {$\eta_S$};
        \draw[mu1,->] (-0.5,1) -- (3.7,1) node[pos=0.5,above] {\small labeling of branch points};
        \draw[mu1,->] (6.3,1) -- (8.7,1) node[pos=0.5,above] {\small period matrix};
        \draw[mu1,|->] (-0.5,0) -- (3.7,0);
        \draw[mu1,|->] (6.3,0) -- (8.7,0);
\end{tikzpicture}
\end{center}
\vspace{0.5 cm}
where above each map we have noted the choice made to give the map. The resulting class of maps $\eta$ is given by the composition of these two isomorphisms.

\begin{definition}
	We say that the class of the map $\eta \in \Xi_g$ is associated to the period matrix $Z \in \mathcal{H}_g$ if there is a labeling of the branch points of the hyperelliptic map such that for all $S\subseteq B$ with $\# S \equiv 0 \pmod{2}$, we have $AJ(e_S)=(\eta_S)_2 + Z (\eta_S)_1$, where the Abel-Jacobi map is defined by the symplectic basis used to compute the period matrix $Z$.
\end{definition}

\begin{example}[Mumford]\label{MumfordAzygeticSystem}
	In \cite[Chapter 5]{Mumford}, Mumford chooses an explicit symplectic basis for homology and computes the associated class in $\Xi_g$. For the convenience of the reader, a drawing of his basis is given in Appendix \ref{pathpicture}. He obtains the values
	\begin{gather*}
	\tilde{\eta}_{2i-1} = \bigg( \begin{matrix} 0 \ldots 0 \smalloverbrace{\frac{1}{2}}^{i} 0 \ldots 0 \smallunderbrace{\frac{1}{2}}_{g+1} \dfrac{1}{2} \ldots \dfrac{1}{2} \smallunderbrace{0}_{g + i} 0 \ldots 0 \end{matrix} \bigg) \quad \text{for $i = 1, \ldots g+1$ and}\\
	\tilde{\eta}_{2i} = \bigg( \begin{matrix} 0 \ldots 0 \smalloverbrace{\frac{1}{2}}^{i} 0 \ldots 0 \smallunderbrace{\frac{1}{2}}_{g+1} \dfrac{1}{2} \ldots \dfrac{1}{2} \smallunderbrace{\frac{1}{2}}_{g + i} 0 \ldots 0 \end{matrix} \bigg) \quad \text{for $i = 1, \ldots g$.}
	\end{gather*}
	 One can show using Proposition \ref{prop:asygetic} that this indeed gives rise to a class of maps in $\Xi_g$, which we denote by $\tilde{\eta}$ throughout the paper.
	
	The set $U_{\tilde{\eta}}$ associated to this map is $\{2, 4, \ldots 2g, \infty\}$.
	He also computes
    \begin{equation*}    \tilde{\eta}_{U_{\tilde{\eta}}} = \bigg( \begin{matrix} \dfrac{1}{2}   \ldots  \dfrac{1}{2} & \dfrac{g}{2} & \dfrac{g-1}{2} \ldots \dfrac{3}{2} & 1 & \dfrac{1}{2}  \end{matrix} \bigg).
    \end{equation*}
\end{example}

\vspace*{0.7 cm}
We denote by $\Sp_{2g}(\Z)$ the group of $2g \times 2g$ matrices symplectic with respect to the bilinear form $A(x,y) = x_1^T  y_2 - x_2^T  y_1$ and with coefficients in $\mathbb{Z}$.
There is an action of the group $\Sp_{2g}(\mathbb{Z})$ on the set $\Xi_g$ given by matrix multiplication on the left on the codomain of a map $\eta$.  We also define
\begin{equation*}
\Gamma_2 = \{\gamma \in \Sp_{2g}(\Z) : \gamma \equiv \mathds{1}_{2g} \pmod{2}\},
\end{equation*}
the principal congruence subgroup of level $2$.

We have the following:
\begin{proposition}[Igusa {\cite[Chapter V, Section 6]{igusa}}]\label{prop:igusa}
	The quotient group $\Sp_{2g}(\mathbb{Z})/\Gamma_2$ acts freely and transitively on the classes in $\Xi_g$.
\end{proposition}
\begin{proof}	Using the isomorphism $(1/2)\mathbb{Z}^{2g}/\Z^{2g} \cong \mathbb{F}_2^{2g}$, then $\Sp_{2g}(\mathbb{Z})/\Gamma_2 \cong \Sp_{2g}(\mathbb{F}_2)$ is exactly the group of symplectic matrices for the non-degenerate bilinear mapping given by $e(x,y) = (-1)^{x_1^T  y_2 - x_2^T  y_1}$ for $x, y \in \mathbb{F}_2^{2g}$.
\end{proof}
Thanks to this action, we need only one example of a class of maps $\eta \in \Xi_g$ to obtain a representative from each class, which is given to us by Mumford's explicit computation.

\section{The Vanishing Criterion and Thomae's formulae}\label{vanishing}

\subsection{Theta functions, theta characteristics, and theta constants}

For $\omega \in \C^g$ and $Z \in \mathcal{H}_g$, we define the following important theta series:
\begin{equation}
\vartheta(\omega, Z) = \sum_{n \in \Z^{g}}\exp(\pi i n^T Z n + 2 \pi i n^ T \omega).
\end{equation}

Recall that given a period matrix $Z \in \mathcal{H}_g$, we denote by $L_Z$ the lattice that is generated by the columns of the matrix $(\mathds{1}_g,Z)$. This gives a set of coordinates on the torus $\C^g/L_Z$ in the following way: A vector $x \in [0,1]^{2g}$ gives the point $x_2 +  Z x_1 \in \C/L_Z$, where, as in the previous section, $x_1$ denotes the first $g$ entries and $x_2$ denotes the last $g$ entries of a vector of length $2g$.

Of interest to us will be the values of $\vartheta(\omega, Z)$ at points $\omega \in \C^{g}$ that, under the natural quotient map $\C^g \to \C^g/L_Z$, map to $2$-torsion points. These points are of the form $\omega = \xi_2 + Z \xi_1$ for $\xi \in (1/2)\Z^{2g}$. This motivates the following definition:
\begin{equation*}
\vartheta[\xi](Z) = \exp(\pi i \xi_1^T Z \xi_1 + 2 \pi i \xi_1^T \xi_2) \vartheta( \xi_2 + Z \xi_1, Z).
\end{equation*}
In this context, $\xi$ is customarily called a \emph{characteristic} or \emph{theta characteristic}. The value $\vartheta[\xi](Z)$ is called a \emph{theta constant}. It is the special value $\vartheta[\xi](0,Z)$ of the theta function with characteristic $\xi$, which is defined in \cite[p. 123]{Mumford1}.

\begin{definition}
	We say that a characteristic $\xi \in (1/2)\Z^{2g}$ is \emph{even} if $e_*(\xi) = 1$ and \emph{odd} if $e_*(\xi) = -1$. If $\xi$ is even we call $\vartheta[\xi](Z)$ an \emph{even theta constant}, and if $\xi$ is odd we call $\vartheta[\xi](Z)$ an \emph{odd theta constant}.
\end{definition}

We have the following fact about the series $\vartheta[\xi](\omega,Z)$ \cite[Chapter II, Proposition 3.14]{Mumford1}: For $\xi \in (1/2)\Z^{2g}$,
\begin{equation*}
\vartheta[\xi](-\omega,Z) = e_*(\xi) \vartheta[\xi](\omega,Z).
\end{equation*}
From this we conclude that all odd theta constants vanish.

Finally, we will most often only be concerned about the vanishing or non-vanishing of certain values $\vartheta[\xi](Z)$ for $\xi$ even. In this case, because when $n \in \Z^{2g}$ we have
\begin{equation*}
\vartheta[\xi + n](Z) = \exp(2 \pi i \xi_1^T n_2)\vartheta[\xi](Z),
\end{equation*} 
we note that the vanishing depends only on the equivalence class of $\xi$ in $(1/2)\Z^{2g}/\Z^{2g}$.

\subsection{The Vanishing Criterion}
We are finally in a position to state the Mumford-Poor Vanishing Criterion:
\begin{theorem}[Poor {\cite[Main Theorem 2.6.1]{poor}}]\label{vanishingtheorem}
	Let $Z \in \mathcal{H}_g$ and $\eta \in \Xi_g$. Then the following statements are equivalent:
	\begin{itemize}
		\item $Z$ is the period matrix of a symplectically irreducible abelian variety and satisfies the following equations for a map $\eta \in \Xi_g$:
		\begin{equation}\label{VanishingCondition}
	 \text{For} ~ S \subseteq B ~ \text{with} ~ \#S\equiv 0 ~ \pmod{2}, \vartheta[\eta_S](Z) = 0 ~ \text{if and only if} ~ \#(S\circ U_{\eta}) \neq g+1.
		\end{equation}
		\item There is a hyperelliptic curve of genus $g$ whose Jacobian has period matrix $Z$ and $\eta$ is one of the maps associated to $Z$.
	\end{itemize}	
\end{theorem}

\begin{remark}\label{symplecticallyirr}
Poor defines \emph{symplectically irreducible} on page 831 of \cite{poor}. His condition is equivalent to requiring that the abelian variety is not isomorphic as a polarized abelian variety to a product of lower-dimensional polarized abelian varieties. In this work, our period matrices are constructed to be \emph{simple}, i.e., not isogenous to a product of lower-dimensional polarized abelian varieties. Since isomorphism is stronger than isogeny, all of the period matrices we construct are symplectically irreducible, and we may apply the theorem.
\end{remark}

We note that the original idea behind the Vanishing Criterion is due to Mumford \cite{Mumford}. In Chapter 3, Corollary 6.7 of \emph{loc. cit.}, Mumford shows that with his specific choice of symplectic basis for the homology group of the hyperelliptic curve, the period matrix obtained satisfies the vanishing criterion (\ref{VanishingCondition}) above. In Theorem 9.1, he then presents a partial converse and states that if there is a map $\eta$ whose distinguished set $U_{\eta}$ has $g+1$ elements such that $Z$ satisfies the vanishing criterion (\ref{VanishingCondition}) for the map $\eta$, then $Z$ is a hyperelliptic period matrix.

We state Poor's result above because Theorem 9.1 in~\cite{Mumford} does not cover every hyperelliptic period matrix, as shown by the following example.

\begin{example}\label{etabar}
	Consider Mumford's choice of symplectic basis and his choice of labeling for the branch points of the hyperelliptic curve, exhibited at the beginning of \cite[Chapter 5]{Mumford} and reproduced in Appendix \ref{pathpicture} below. Act on this basis by the symplectic matrix
\begin{equation}\label{eq:gammabar}
\bar{\gamma} =
\begin{pmatrix}
1 & 1 & 1 & -1 & 1 & 1 \\
0 & 1 & 0 & 0 & -1 & 0 \\
0 & 1 & 0 & -1 & 1 & 1 \\
-1 & -1 & -1 & 2 & -1 & -1 \\
0 & 0 & -1 & 1 & -1 & -1 \\
0 & -1 & -1 & -1 & 1 & 0
\end{pmatrix},
\end{equation}
without changing the labeling of the points. As computed by Poor \cite{poor}, if $\gamma = \left(\begin{smallmatrix}A & B \\ C & D \end{smallmatrix}\right) \in \Sp_{2g}(\Z)$ acts on a symplectic basis with associated map $\tilde{\eta}$, then after the change of basis, if the branch points are not relabeled, the map associated to the new basis will be given by $\gamma^* \tilde{\eta}$, where $\gamma^* = \left(\begin{smallmatrix}A & -B \\ -C & D \end{smallmatrix}\right)$. In our particular case, after performing this action, the class of the new map $\bar{\eta} = \gamma^* \tilde{\eta}$ is given by the values
\begin{gather*}
\bar{\eta}_1 = \begin{pmatrix} \frac{1}{2}  & 0 & 0 & \frac{1}{2}  & 0 & 0\end{pmatrix}, \quad \bar{\eta}_2 = \begin{pmatrix} 0 & 0 & \frac{1}{2}  & \frac{1}{2}  & \frac{1}{2}  & \frac{1}{2} \end{pmatrix} \\
\bar{\eta}_3 = \begin{pmatrix} 0 & \frac{1}{2}  & 0 & \frac{1}{2}  & \frac{1}{2}  & 0\end{pmatrix}, \quad \bar{\eta}_4 = \begin{pmatrix} \frac{1}{2}  & 0 & \frac{1}{2}  & 0 & 0 & \frac{1}{2} \end{pmatrix} \\
\bar{\eta}_5 = \begin{pmatrix} \frac{1}{2}  & \frac{1}{2}  & 0 & 0 & \frac{1}{2}  & \frac{1}{2} \end{pmatrix}, \quad \bar{\eta}_6 = \begin{pmatrix} 0 & \frac{1}{2}  & \frac{1}{2}  & \frac{1}{2}  & 0 & \frac{1}{2} \end{pmatrix} \\
\bar{\eta}_7 = \begin{pmatrix} \frac{1}{2}  & \frac{1}{2}  & \frac{1}{2}  & 0 & \frac{1}{2}  & 0\end{pmatrix},
\end{gather*}
where each entry is reduced modulo $\Z$. The map $\bar{\eta}$ given by these values has distinguished set $U_{\bar{\eta}}$ equal to all of $B$, which does not have cardinality $g+1 = 4$. Furthermore, we will show in Lemma \ref{lemma:sizeU} that if a period matrix is associated to $\bar{\eta}$, then it will only be associated to maps $\eta$ with $\# U_{\eta} = 8$.
\end{example}

\subsection{Takase's modified formula}

Given a hyperelliptic period matrix $Z$ and one of its associated maps $\eta$, we can construct a model for the hyperelliptic curve via Thomae's formulae. To state the formulae, we set up some notation. If $Z$ is a symplectically irreducible period matrix satisfying the Vanishing Criterion (\ref{VanishingCondition}) for some map $\eta \in \Xi_g$, then $Z$ is the period matrix of a hyperelliptic Jacobian. Further, the map $\eta$ comes equipped with a labeling of the branch points of the hyperelliptic map $\pi \colon X \to \mathbb{P}^1$, $P_1, \ldots, P_{2g+1}, P_{\infty}$. Let $x$ be a choice of $x$-coordinate such that the hyperelliptic curve has a model of the form $y^2 = f(x)$, for $f$ of degree $2g+1$, with $x(P_{\infty}) = \infty$. Then we write $a_i = x(P_i)$ for $i = 1, \ldots, 2g+1$, and these are all finite values in $\C$.

\begin{theorem}[Thomae {\cite[Chapter III, Theorem 8.1]{Mumford}}]
	Let $Z$ satisfy the Vanishing Criterion for a map $\eta \in \Xi_g$. Then for all sets $S \subseteq B-\{ \infty \}$, $\#S$ even, and with notation as above, there is a constant $c$ independent of $S$ such that
	\begin{equation*}
		\vartheta[\eta_S](Z)^4=
		\begin{cases}
			0 & \text{if }\#(S\circ U_{\eta}) \neq g+1,\\
			c\cdot (-1)^{(\#S\cap U_{\eta})}\cdot  \prod _{\substack{i\in S\circ U_{\eta}, \\ j\in B-S\circ U_{\eta}-\{\infty\}}}(a_i-a_j)^{-1} & \text{if } \#(S\circ U_{\eta})=g+1.
		\end{cases}
	\end{equation*}
\end{theorem}

\begin{proof}
We note here the slight modifications to Mumford's proof that are necessary to ensure that the proof applies to any period matrix, and not only those considered by Mumford (see the remarks immediately above Example \ref{etabar} for more details).

The proof of Mumford's Theorem 8.1 logically relies on Proposition 6.3, which we assume here to be true about any map $\eta$, and Theorem 7.6. Theorem 7.6 in turn relies on Part 3 of Theorem 5.3 and Corollary 7.4. The proof of part 3 of Theorem 5.3 is valid, as long as $\delta$ is replaced with the vector $\eta_{U_{\eta}}$ for a map $\eta$ associated with the period matrix $Z$ and $\eta_k$ is as in our definitions. The argument of Corollary 7.4 relies only on the Generalized Frobenius Theta Formula (Theorem 7.1), which Mumford shows only for maps $\eta$ with $\#U_{\eta} = g+1$, but which is shown in full generality by Poor \cite[Proposition 1.6.10]{poor}. Therefore we conclude that the Thomae formulae are valid for any period matrix, since the Generalized Frobenius Theta Formula is.
\end{proof}

Then we have the following:
\begin{theorem}[Takase \cite{Takase}]\label{Th:Takase}
Let $Z$ a period matrix and $\eta\in \Xi_g$ be such that the Vanishing Criterion (\ref{VanishingCondition}) is satisfied. Then, again with notation as above, for any disjoint decomposition $B - \{\infty\}=V\sqcup W \sqcup \{k,l,m\}$ with $\# V=\# W=g-1$, we have
	\begin{equation*}
		\frac{a_k-a_l}{a_k-a_m}=\epsilon(k,l,m)\left(\frac{\vartheta[U_{\eta}\circ (V\cup\{k,l\})]\cdot \vartheta[U_{\eta}\circ (W\cup \{k,l\})]}{\vartheta[U_{\eta}\circ (V\cup \{k,m\})]\cdot \vartheta[U_{\eta}\circ (W\cup \{k,m\})]}\right)^2,
	\end{equation*}
	with
	$$\epsilon(k,l,m)=
	\begin{cases}	
	1 & \text{if } k<l,m \text{ or } l,m<k \\
	-1 & \text{if } l<k<m \text{ or } m<k<l, \\
	\end{cases}
	$$
	and where to lighten the notation we denote $\vartheta[\eta_S](Z)$ by $\vartheta[S]$.
\end{theorem}
\begin{proof}
The proof follows as in~\cite{Takase} for any period matrix $Z$, once we replace Mumford's $\tilde{\eta}$ with any $\eta$ associated to our period matrix $Z$ and $U$ with the set $U_{\eta}$.
\end{proof}

\begin{remark}
Note that in genus 2, all vanishing theta constants correspond to odd characteristics, and the conditions in Theorem~\ref{vanishingtheorem} for a given period matrix $Z$ are trivially verified for any map $\eta \in \Xi_2$. Therefore it is possible to give a formula for the Rosenhain invariants of a genus $2$ period matrix that is valid for any period matrix. In consequence the issues considered in this paper, where we are concerned with defining the set $\Xi_g$ and computing an element of it corresponding to a given period matrix, do not arise.
\end{remark}

Finally, to fix a model for our hyperelliptic curve of genus $3$, we require that $x(P_1) = 0$ and $x(P_2) =1$ and compute the Rosenhain model
\begin{equation*}
	y^2=x(x-1)(x-a_3)(x-a_4)(x-a_5)(x-a_6)(x-a_7)
\end{equation*}
of the curve. This allows us to compute $a_i$, $i = 3, \ldots, 7$ directly using the formula above, with the choice $k = 1$ and $m = 2$ for each $i$.

\section{Computing the map $\eta$}\label{computingeta}

We now show how to give a map $\eta$ associated with a period matrix given only the values of the even theta constants. We note that throughout this section, we will be concerned with computing the \emph{class} of a map $\eta$ associated to $Z$. To apply Theorem \ref{Th:Takase}, we then lift each value $\eta_i \pmod{\Z^6}$ to a value in $(1/2)\Z^6$ in the naive way, and then use these values to compute $\eta_S$ for the other $S \subseteq B$ using Property \ref{property:sum} of Definition \ref{def:Xi}.

We have already remarked that with Mumford's map $\tilde{\eta}$ and the transitive action of $\Sp_{2g}(\mathbb{F}_2)$ on the classes of $\Xi_g$, we can compute a representative of each class. It would suffice then to verify if $Z$ satisfies the Vanishing Criterion for each class of maps until we find one that works. Unfortunately, the size of the group $\Sp_{2g}(\mathbb{F}_2)$ grows quickly as $g$ grows which makes this unmanageable. For this reason, in this section we provide a faster way to construct a map $\eta$ attached to a period matrix for the case $g=3$. Throughout, we will need the group $\Gamma_{1,2}$, where for $Q$ the quadratic form $Q(x) = x_1^T  x_2$,
 \begin{equation*} \Gamma_{1,2} = \{\gamma \in \Sp_{2g}(\Z) : Q(\gamma x) \equiv Q(x) \pmod{2} \}.
 \end{equation*}
We note that the quotient $\Gamma_{1,2}/\Gamma_2$ is isomorphic to the special orthogonal group of matrices that preserve the $2g$-ary positive definite quadratic form of Arf invariant $0$ over $\mathbb{F}_2$. (See \cite[Equation (1.10)]{grossharris} for the definition of Arf invariant, as well as a discussion of the relationship between quadratic forms and theta characteristics.)

\subsection{Characterizing hyperelliptic Jacobians when $g = 3$}
To recognize hyperelliptic Jacobians in our situation, we use the following theorem:

\begin{theorem}[Igusa {\cite[Lemmata 10 and 11]{igusa67}}] When $g = 3$, the Vanishing Criterion (\ref{VanishingCondition}) reduces to the following: If $Z$ is the period matrix of a simple ppav of dimension $3$, then $Z$ is the period matrix of a hyperelliptic Jacobian if and only if $\vartheta[\xi](Z)$ vanishes for a single equivalence class $\xi \in (1/2)\Z^{6}/\Z^{6}$ with $e_*(\xi) = 1$.
\end{theorem}

\begin{proof}
Let $\xi_i \in (1/2)\Z^{6}/\Z^{6}$ be an ordered set of representatives of the equivalence classes of even vectors $\xi \in (1/2)\Z^{6}$, where the ordering is arbitrary. We note that up to equivalence modulo $\Z^6$, there are $36$ even theta characteristics.

In \cite{igusa67}, Igusa defines two distinguished Siegel modular forms of genus $3$,
\begin{equation*}
\Sigma_{140}(Z) = \sum_{i = 1}^{36} \prod_{j \neq i} \vartheta[\xi_j](Z)^8,
\end{equation*}
and
\begin{equation*}
\chi_{18}(Z) = \prod_{i=1}^{36} \vartheta[\xi_i](Z),
\end{equation*}
and shows that $\Sigma_{140}(Z)$ vanishes exactly on the locus of period matrices $Z$ that are symplectically reducible (this is equivalent to requiring that the associated polarized abelian variety is isomorphic to a product of lower-dimensional polarized abelian varieties), and $\chi_{18}(Z)$ vanishes on the locus of period matrices $Z$ that are symplectically reducible or whose associated ppav is isomorphic to the Jacobian of a hyperelliptic curve.

	It now suffices to notice that if $Z$ is the period matrix of a simple ppav of dimension $3$, then $\Sigma_{140}(Z) \neq 0$, so at most one even theta constant vanishes. But in this case,  $Z$ is hyperelliptic if and only if $\chi_{18}(Z)$ vanishes, which implies that at least one even theta constant vanishes.
\end{proof}

\begin{definition}
	Let $Z$ be the period matrix of a simple genus $3$ hyperelliptic Jacobian. Then we denote by $\delta$ the unique vector in $(1/2)\Z^{6}/\Z^{6}$ such that $\vartheta[\delta](Z) =0$ and $e_*(\delta) = 1$, and call it the \emph{vanishing even characteristic}.
\end{definition}

\begin{proposition}\label{prop:etaU}
	Suppose that $Z$ is the period matrix of a simple hyperelliptic Jacobian and $g = 3$. Then for any $\eta$ associated to $Z$, the vanishing even characteristic of $Z$ is $\eta_{U_{\eta}}$. Conversely, if $Z$ has vanishing even characteristic $\delta$ and $\delta = \eta_{U_{\eta}}$ for some map $\eta$, then $Z$ satisfies the Vanishing Criterion (\ref{VanishingCondition}) for the map $\eta$.
\end{proposition}

\begin{proof}
	Let $\eta$ be a map associated to $Z$, with distinguished set $U_{\eta}$. Because $\#(U_{\eta} \circ U_{\eta}) = 0$, by part \ref{property:parity} of Definition \ref{def:Xi}, $e_*(\eta_{U_{\eta}})= (-1)^{4/2} = 1$ and $\eta_{U_{\eta}}$ is an even characteristic. We also have that $\# U_{\eta} = 4$ or $8$ and $\#(U_{\eta} \circ U_{\eta})\neq 4$ so $\vartheta[\eta_{U_{\eta}}](Z) = 0$ by the Vanishing Criterion. Therefore $\eta_{U_{\eta}}$ is the unique vanishing even theta constant.
	
	Now for the converse, suppose that there is a map $\eta$ with $\eta_{U_{\eta}} = \delta$, where $\delta$ is the unique vanishing even characteristic of $Z$. Then we need to show that for any $S$ of even cardinality such that $\#(S\circ U_{\eta}) \neq g+1 =4$, $\vartheta[\eta_S](Z) = 0$.
	
	Because $\#(S \circ U_{\eta}) = \# U_{\eta}  + \# S - 2 \# (S \cap U_{\eta})$, the possibilities for $\#(S \circ U_{\eta})$, excluding $\#(S \circ U_{\eta}) = 4$, are $0$, $2$, $6$ or $8$. If $\#(S \circ U_{\eta}) = 2$ or $6$, then $e_*(\eta_S) = (-1)^{(4 -  \#(U_{\eta} \circ S) )/2} = -1$ by property \ref{property:parity} of Definition \ref{def:Xi}, and so $\vartheta[\eta_S](Z) = 0$ because it is an odd theta constant.
	
	In the case where $\#(S \circ U_{\eta}) = 0$ or $8$, we must have $S = U_{\eta}$ or $S = U_{\eta}^c$, respectively. In that case $\eta_{S} = \eta_{U_{\eta}} = \delta$ and $\vartheta[\eta_S](Z) = 0$ by assumption.
\end{proof}

\subsection{Computing the maps $\eta$}
There are two cases to this computation: We first consider the case where $\delta \neq 0$, and then the case where $\delta = 0$.

\begin{lemma}\label{lemma:sizeU}	Let $Z$ be the period matrix of a simple genus $3$ hyperelliptic Jacobian. Then $\vartheta[0](Z) = 0$ if and only if for every map $\eta$ associated to $Z$, $\# U_{\eta} =  8$.
\end{lemma}

\begin{proof} For $\xi$ even and any map $\eta$ associated to $Z$, $\vartheta[\xi](Z) = 0$ if and only if $\xi = \eta_{U_{\eta}}$. In turn, $\eta_{S_1} = \eta_{S_2}$ if and only if $S_1 = S_2$ or $S_1 = S_2^c$, and $\eta_{\emptyset} = 0$. This forces $\# U_{\eta} = 0$ or $8$, but since $\infty \in U_{\eta}$, $\#U_{\eta} = 8$.\end{proof}

\begin{remark}
	This shows that when the map $\bar{\eta}$ from Example \ref{etabar} is associated to a period matrix $Z$, then the vanishing even characteristic of $Z$ will be $\delta = 0$, which forces every other map $\eta$ associated to $Z$ to have $\# U_{\eta} = 8$.
\end{remark}

\begin{proposition}\label{prop:cardinality4}
	Suppose that $Z\in \mathcal{H}_3$ is the period matrix of a simple genus 3 hyperelliptic Jacobian and satisfies $\vartheta[\delta](Z) = 0$ for exactly one even characteristic, and $\delta \neq 0$. Then there is $\gamma \in \Gamma_{1,2}$ such that $Z$ satisfies the Vanishing Condition for the map $\eta = \gamma \tilde{\eta}$. Furthermore, this $\gamma$ can be taken to be any such that $\gamma (\begin{smallmatrix} \frac{1}{2}  & \frac{1}{2}  & \frac{1}{2}  & \frac{1}{2}  & 0 & \frac{1}{2}  \end{smallmatrix}) = \delta \pmod{\Z^6}$.
\end{proposition}

\begin{proof}
	Since $Z$ is the period matrix of a hyperelliptic Jacobian, there are several maps in $\Xi_g$ such that $Z$ satisfies the Vanishing Condition for these maps. Choose any such and denote it $\eta^*$. Then the cardinality of the distinguished set $U_{\eta^*}$ is $4$ by Lemma \ref{lemma:sizeU}.
	
	Recall that $\eta^*_i = AJ(P_i - P_{\infty})$, by equation (\ref{eq:2torsion}). Relabel the points $\{P_1,\ldots,P_7\}$ so that if $e_*(AJ(P_i - P_{\infty})) = 1$, then $i \in \{1,3,5,7\}$ and if $e_*(AJ(P_i - P_{\infty})) = -1$ then $i \in \{2,4,6\}$. This is possible because exactly three values of $i \in \{1,\ldots,7\}$ are such that $e_*(\eta^*_i) = -1$. This relabelling gives rise to a different map $\eta$ which is still associated to $Z$.
	
	We now have $e_*(\eta_i) = e_*(\tilde{\eta}_i)$ for each $i$, and there is $\gamma\in \Sp_{6}(\F_2)$ with $\gamma \tilde{\eta} = \eta \pmod{\Z^6}$. We show that in fact $\gamma \in \Gamma_{1,2}/\Gamma_2$ by showing that for all $\xi \in (1/2)\Z^6/\Z^6$, $e_*(\gamma \xi) = e_*(\xi)$.
	
	For any class $\eta \in \Xi_3$, the values $\eta_i$ for $i = 1, \ldots 6$ form a basis of the $\mathbb{F}_2$-vector space $(1/2)\Z^{6}/\Z^{6}$. Therefore any $\xi \in (1/2)\Z^{6}/\Z^{6}$ can be written as a sum of elements in this basis, say $\xi = \sum_{k \in S} \tilde{\eta}_k$, and since $e_2(\tilde{\eta}_i,\tilde{\eta}_j) = -1$ whenever $i \neq j$,
	\begin{equation*}
	e_*(\xi) = (-1)^{\binom{\# S}{2}} \prod_{k\in S} e_*(\tilde{\eta}_k).
	\end{equation*}
	On the other hand, $\gamma \xi = \gamma\sum_{k \in S} \tilde{\eta}_k = \sum_{k \in S} \gamma\tilde{\eta}_k$ and applying the same argument to the map $\eta$, we have
	\begin{equation*}
	e_*(\gamma\xi) = (-1)^{\binom{\# S}{2}} \prod_{k\in S} e_*(\gamma\tilde{\eta}_k).
	\end{equation*}
	But $e_*(\gamma\tilde{\eta}_k) = e_*(\eta_k) = e_*(\tilde{\eta}_k)$ by assumption and so $e_*(\gamma\xi)=e_*(\xi)$ and $\gamma \in \Gamma_{1,2}$.	
	
	We also have
	\begin{equation*}
	\eta_U = \sum_{i \in U_{\eta}} \eta_i = \sum_{i \in U_{\tilde{\eta}}} \gamma \tilde{\eta_i} = \gamma \sum_{i \in U_{\tilde{\eta}}}\tilde{\eta_i} = \gamma (\begin{smallmatrix} \frac{1}{2}  & \frac{1}{2}  & \frac{1}{2}  & \frac{1}{2}  & 0 & \frac{1}{2}  \end{smallmatrix}) = \delta,
	\end{equation*}
	which completes the proof.
\end{proof}

 \begin{proposition}
 	Suppose that $Z\in \mathcal{H}_3$ is irreducible and satisfies $\vartheta[0](Z) = 0$. Then $Z$ satisfies the Vanishing Criterion for $\bar{\eta}$ defined in Example \ref{etabar}.
 \end{proposition}

 \begin{proof}
 	By Proposition \ref{prop:etaU}, it suffices to show that for the map $\bar{\eta}$, $\bar{\eta}_{U_{\bar{\eta}}} = 0$. This follows since $U_{\bar{\eta}} = B$, and $\sum_{i \in B} \bar{\eta}_i = 0$.
 \end{proof}

\begin{remark} \label{compareWeng}
In~\cite{WengGenus3}, Weng computes the Rosenhain invariants of hyperelliptic curves using formulae similar to those of Takase. In particular, her formulae also depend on the $\Gamma_2$-equivalence class of the period matrix. In this article, we give this dependence in terms of a map $\eta$ associated to the period matrix, which is computed as $\eta = \gamma \tilde{\eta}$ for a matrix $\gamma \in \Sp_6(\Z)$ that depends on the vanishing even theta constant of the period matrix. Weng instead gives an explicit table (\cite[Table 1]{WengGenus3}) of how her formulae should be modified to account for the vanishing even theta constant. This table enumerates the specific theta constants that must be substituted into the formulae giving each of the five Rosenhain invariants, for each of the thirty-six possible vanishing even theta constants. Although no mention is given in her paper of how this table was obtained, we believe that she uses the approach given by Weber~\cite{weber}.

We were unable to use this table in our computations for two reasons. The first is that both Weber and Weng define the set $U$ to be $\{1,3,5,7\}$, without mention of its dependence on the map $\eta$. (We note that since $\eta_{S} = \eta_{S^c}$, this is equivalent to $U = \{2, 4, 6, \infty\}$, which is the convention that we have adopted in this paper.) As is shown in the proof of Proposition \ref{prop:cardinality4}, often $\eta$ can be chosen such that $U_{\eta} = \{1,3,5,7\}$, but there is no mention by Weber that the asygetic system used in the computation has this property. Secondly, this is not always possible, as demonstrated in Example \ref{etabar}.
It may be that these considerations were taken into account in the computation of the table given by Weng, and that this was simply omitted in the text. However, since we could not be sure of it, using the table would have involved verifying each of its $540$ entries. Instead, we chose to provide the reader with a proof of correctness for the whole method.
\end{remark}

\section{Implementation, examples and results}\label{results}

We implemented the algorithms described here in Sage and PARI/GP; our code is available at \cite{github}. Our search for hyperelliptic curves and their construction uses Algorithm~\ref{ConstructionCurves}.

\begin{algorithm}[h]
\caption{Computing Rosenhain coefficients}
\label{ConstructionCurves}
\begin{algorithmic}[1]
\Require A sextic CM field $K$, and precision $prec$
\Ensure Rosenhain coefficients for all hyperelliptic curves with CM by $K$, if any exist.
\State Compute all period matrices $Z$ having CM by $K$ (with precision $prec$) using Algorithm~\ref{alg:periodmatrix}.
\For{each period matrix $Z$}
	\State Compute all even theta constants and the set $T$ of characteristics for which the theta constants are $\leq 10^{prec}$.
	\If{$T$ has exactly one element $\delta$}
		\If{$\delta \neq 0$}
			\State Compute $\gamma \in \Gamma_{1,2}$ such that $\delta=\gamma (\begin{smallmatrix} \frac{1}{2}  & \frac{1}{2}  & \frac{1}{2}  & \frac{1}{2}  & 0 & \frac{1}{2}  \end{smallmatrix})$. This can be computed once and for all and stored for each of the $35$ possible $\delta$s.
		\ElsIf{$\delta = 0$}
			\State Let $\gamma = \bar{\gamma}$ from Equation~\eqref{eq:gammabar}
		\EndIf
		\State Compute the Rosenhain coefficients with precision $prec$ using the formulae from Theorem~\ref{Th:Takase}, and $\eta_i=\gamma \tilde{\eta_i}$, for all $i\in \{3,\ldots, 7\}$.
\EndIf
\EndFor
\end{algorithmic}
\end{algorithm}

The software implements the different steps in Algorithm~\ref{ConstructionCurves} as follows:

\begin{itemize}
\item Algorithm~\ref{alg:periodmatrix} is implemented in Sage \cite{Sage} and all computations are done symbolically. The running time of this step is negligible compared to the following steps.

\item The computation of theta constants is performed by a PARI/GP program. This is the most time-consuming part of the algorithm. Indeed, in order to compute a theta constant, we approximate
$\theta[\xi](Z)$ by
\begin{equation*}
S_{\xi,B}=\sum _{n\in [-B,B]^3} \exp \left (\pi i((n+\frac{1}{2}\xi_1)^T Z(n+\frac{1}{2}\xi_1)+2(n+\frac{1}{2}\xi_1)^T(\frac{1}{2}\xi_2))\right ),
\end{equation*}
with $B>0$. To ensure that our computation is correct up to $N$ bits of precision, we would need to estimate the error bound as a function of $B$ and $N$. In genera 1 and 2, this was previously done by computing with period matrices in the fundamental domain (see~\cite{Gottschling,Dupont}). In genus 3, no method for computing matrices in the fundamental domain is known. To make sure we computed correctly with precision $t$, we computed $S_{\xi,B}$ for several values of $B$ until we obtained $|S_{\xi,B'}-S_{\xi,B}|<2^{-t}$ for two consecutive values $B'>B$.

\item To recognize the values of the Rosenhain coefficients $a_i$ as algebraic integers we use the algebraic dependence testing algorithm~\cite{Cohen}, implemented in PARI/GP by the function \texttt{algdep}. We obtain a conjectured minimal polynomial $\lambda_i$ for each coefficient $a_i$. Note that the amount of precision needed for this computation to end successfully depends on the dimension of the lattice fed to \texttt{algdep}, i.e., on the degree of the minimal polynomial of the Rosenhain coefficients. This degree depends on the class number of $K$ (see~\cite{Costello} for details). In practice, since we only computed with sextic fields of class number one, 53 bits of precision sufficed, and the degrees of the polynomials were at most $12$. However, we expect the amount of precision needed for this computation to increase dramatically once the class number of $K$ is increased.

\item In order to heuristically check the correctness of our Rosenhain minimal polynomials, we choose a prime $p$ such that there is an unramified prime ideal $\mathfrak{p}$ of degree 1 over $p$ in $\mathcal{O}_{K^r}$ (the ring of integers of the reflex field) and such that there is $\pi\in K$ with $\pi\bar{\pi}=p$.  
By a theorem of Shimura~\cite[Theorem 2 in Sect. 13]{Shimura}, the reduction of an abelian variety with CM by $\mathcal{O}_K$ is an abelian variety over $\F_p$ with maximal complex multiplication. We compute Weil numbers $\pi$ which correspond to the Frobenius endomorphism on a curve defined over $\F_p $ whose Jacobian has CM by $\mathcal{O}_K$. The Rosenhain invariants of such the hyperelliptic curve should be roots of the polynomials we have computed (modulo $p$). So we loop through all roots of the $5$ minimal polynomials until we find a curve whose Jacobian has the right number of points.  Specifically, let $n$ be the degree of the Rosenhain polynomials (in all our class number one computations this was 3, 6 or 12). We construct all $n^5$ curves obtained in this way and check whether the Jacobian of the curve has cardinality equal to $N_{K/\Q}(1-\pi)$. For a higher degree of certainty one can compute the zeta function, but when $p$ is large enough the heuristic check we use is unlikely to coincidentally give the correct cardinality unless the minimal polynomials of the Rosenhain invariants that we have computed are correct.
\end{itemize}

\subsection{Galois CM sextic fields with class number 1 giving hyperelliptic curves}\label{list}

There are $17$ sextic CM fields $K$ with class number 1 that are Galois over $\mathbb{Q}$.  Of these, four admit a hyperelliptic Jacobian. They are as follows:
\begin{enumerate}
\item $K = \Q(\zeta_7)$
\item \label{field2} $K = \Q[X]/(X^6 + 5X^4 + 6X^2 + 1)$
\item $K = \Q[X]/(X^6 + 6X^4 + 9X^2 + 1)$
\item $K = \Q[X]/(X^6 + 13X^4 + 50X^2 + 49)$
\end{enumerate}

For each of these fields, there is a single isomorphism class of ppav with CM by $K$. All of these examples were found by Weng \cite{WengGenus3} or, in the case of $\Q(\zeta_7)$, have long been known. We conjecture that these four examples are all of the hyperelliptic curves with CM by a Galois sextic field having class number one.

\begin{example}\label{Ex:Weng}
	Let $K$ be field number \ref{field2} above.
	The tuple $(\lambda_3(x),\lambda_4(x),\lambda_5(x),\lambda_6(x),\lambda_7(x))$ of minimal polynomials for the Rosenhain coefficients is:
	\begin{align*}
	&(x^3 + 22x^2 - 16x - 8, x^3 - 4x^2 + 3x + 1, -8x^3 + 8x^2 + 2x -1,\\
	&\;\; x^3 - 9x^2 - x + 1,x^3 + 2x^2 - x - 1).
	\end{align*}

For this field, Weng computed the minimal polynomials of the Shioda invariants (the class polynomials) $(h_1(x),h_2(x),h_3(x),h_4(x),h_5(x)).$ These polynomials have degree one, but their coefficients are larger than those of the minimal polynomials of the Rosenhain invariants: \begin{align*}&(1048576x-2187, 131072x-24373629, 16384x+11632436487,\\&\;\; 16384000000000x+2952169653573, 2048000000000000x-1168038669244419).\end{align*}
\end{example}

This is an example of a phenomenon that is well-understood in genus $2$. Indeed, as noted in~\cite{Costello} when $K$ is a quartic CM field, the Rosenhain invariants of an abelian surface with CM by $K$ are defined over $CM_K(2)$, the class field corresponding to the quotient group of fractional ideals of the reflex field $K^r$ which are prime to 2, modulo fractional ideals $\mathfrak{b}$, with the property that $N_{\Phi}(\mathfrak{b})=\alpha\mathcal{O}_K,~\alpha\equiv 1 \pmod 2, \alpha\bar{\alpha}=N_{K/\Q}(\mathfrak{b})$ (where $N_{\Phi}$ is the typenorm corresponding to the CM type $\Phi$). The Igusa invariants, on the other hand, are defined over a class field $CM_K(1)$ of modulus 1. An analysis of the exact numerical relationship between the degrees of these extension fields in this setting is given in \emph{loc. cit.}

As supported by the data of Example~\ref{Ex:Weng}, we expect that similarly in the sextic case the degrees of the Rosenhain polynomials are larger than those of the Shioda polynomials, but that their coefficients are smaller. This is however beyond the scope of this paper.

\begin{example}
	Let $K = \Q(\zeta_7)$. This example is classical: One can compute that there is a single ppav which is simple over $\mathbb{C}$ and with CM by the full ring of integers of $\Q(\zeta_7)$, up to isomorphism over $\mathbb{C}$. This abelian variety is a hyperelliptic Jacobian, and a model for the hyperelliptic curve is given in~\cite{Shimura}.

	We obtain the tuple of Rosenhain minimal polynomials
	\begin{align*}&(x^6 - 5x^5 + 11x^4 - 13x^3 + 9x^2 - 3x + 1, x^6 - 2x^5 + 4x^4 - 8x^3 + 9x^2 - 4x + 1,\\
	&\;\; x^6 - x^5 + x^4 - x^3 + x^2 - x + 1, x^6 - 3x^5 + 9x^4 - 13x^3 + 11x^2 - 5x + 1, \\
	&\;\; x^6 - 4x^5 + 9x^4 - 8x^3 + 4x^2 - 2x + 1).\end{align*}
\end{example}

Despite our focus in this work on $K$ Galois of class number $1$, our algorithm works for $K$ with Galois closure $L$ with any $\Gal(L/\Q)$, and $K$ of any class number. Here we show an example with $\Gal(L/\Q) \cong \Z/2\Z \times S_3$ and $K$ is of class number $1$:
\begin{example}
Let $K = \Q[X]/(X^6 + 9X^4 + 18X^2 + 1)$. Then $\Gal(L/\Q) \cong \Z/2\Z \times S_3$.
 We obtain the following tuple of  minimal polynomials for the Rosenhain coefficients:
\begin{align*}(x^3 - &69x^2 + 198x - 49, 8x^3 - 448x^2 - 2042x + 2401, x^3 - 43x^2 - 606x - 441,\\
&x^3 - 169x^2 + 6479x + 2401, x^3 - 58x^2 - 96x + 72).\end{align*}
\end{example}

\section*{Acknowledgements}\label{ackref}
This collaboration started at the workshop Sage Days 42: Women in Sage, and we are grateful to Sarah Chisholm for her contribution toward an early version of the algorithm. We thank L\'eo Ducas and Enea Milio for helpful discussions and the anonymous reviewers of the ANTS 2016 conference for their comments. The second author thanks Microsoft Research and the University of California, San Diego for hospitality during a visit when part of this work was performed. We also thank ICERM for facilitating a productive working environment which allowed us to complete this project. This work was supported in part by NSF grant DMS-1103831 and the French ANR Project ANR-12-INSE-0014 SIMPATIC.

\appendix
\section{An algorithm for computing period matrices\\ of abelian varieties with CM}\label{periodmatrixalgorithm}

It is well known that all abelian varieties with CM by a given field $K$ have complex points given by $\C^g/\Phi(\mathfrak{a})$, for $\mathfrak{a}$ a fractional ideal of $K$ and $\Phi$ a CM type. To identify which ones are principally polarizable, we must verify if the lattice $\Phi(\mathfrak{a})$ admits a principal polarization. Spallek \cite{Spallek}, based on the work of Shimura and Taniyama, shows that this is the case if and only if there is $\xi \in K$ such that $-\xi^2$ is totally positive in $K_0$, the totally real subfield of $K$ of degree $g$, $\Im(\phi_i(\xi))>0$ for $i = 1, \ldots, g$, and the ideal $(\mathcal{D}_{K/\Q}\mathfrak{a}\mathfrak{\bar{a}})^{-1}$, for $\mathcal{D}_{K/\Q}$ the different of $K$, is principal and generated by $\xi$.

Let $U_{K}$ denote the group of units of $K$, let $U^{+}$ be the subgroup of totally positive units of the group of units of the totally real subfield $K_0$, and let $U_1$ be the subgroup of $U^{+}$ containing only units of the form $\epsilon \bar{\epsilon}$ for $\epsilon \in \O_K^{\times}$. Then to find a suitable $\xi$ given an ideal $\mathfrak{a}$ such that $(\mathcal{D}_{K/\Q}\mathfrak{a}\mathfrak{\bar{a}})^{-1}$ is principal and a generator $b$ of this principal ideal, it is enough to multiply $b$ by a set of coset representatives of $U_K/U^{+}$. If we can find one such suitable $\xi$, all different possibilities -- each giving a different principal polarization -- differ from this first element by an element of the quotient $U^{+}/U_1$, by a theorem of van Wamelen \cite[Theorem 5]{vanwamelen}. In their paper, Koike and Weng \cite{koikeweng} give a procedure to compute representatives for the quotient groups $U_K/U^{+}$ and $U^{+}/U_1$, which we also use.

For completeness, we include this algorithm here since it is part of the code used to obtain the results of this paper.
\begin{algorithm}[h]
	\caption{Generating data for period matrices}
	\label{alg:periodmatrix}
	\begin{algorithmic}[1]
		\Require A sextic CM field $K$
		\Ensure A list of tuples $(\Phi, \mathfrak{a}, \xi)$ for each isomorphism class of simple ppav with CM by $K$
		\State Compute a representative $\Phi$ for each equivalence class of equivalent primitive CM types of $K$
		\State Run through the ideal class group of $K$ and compute a representative $\mathfrak{a}$ for each ideal class such that $(\mathfrak{a}\bar{\mathfrak{a}}\mathcal{D}_{K/\mathbb{Q}})^{-1}$ is principal, and a generator $b$ of the ideal $(\mathfrak{a}\bar{\mathfrak{a}}\mathcal{D}_{K/\mathbb{Q}})^{-1}$
		\For{each triple $(\Phi,\mathfrak{a}, b)$}
			\State Running through representatives $\{u_1, \ldots, u_e\}$ of the finite quotient $U_K/U^{+}$, check if $u_j b$ satisfies the conditions for $\xi = u_j b$ to give a principal polarization for any $j$
			\If{such a $u_j$ can be found, write $\xi = u_j b$}
				\State Output the pairs $(\Phi(\mathfrak{a}), \epsilon_i \xi)$ for $\epsilon_i$ running through representatives of the finite quotient $U^{+}/U_1$. These are the data of all of the nonisomorphic ppav with underlying torus $\C^g/\Phi(\mathfrak{a})$
			\ElsIf{none is found}
				\State There is no ppav with underlying torus $\C^g/\Phi(\mathfrak{a})$
			\EndIf
		\EndFor
	\end{algorithmic}
\end{algorithm}

From the data $(\Phi, \mathfrak{a}, \xi)$, we obtain a period matrix for the ppav by asking Sage for a basis of $\mathfrak{a}$ that is symplectic under the polarization $E_{\xi}$, which is given explicitly in terms of $\xi$ in \cite[p. 19]{lang}. A matrix $[\Omega_1,\Omega_2]$ of size $3 \times 6$ is created by embedding the six elements of this symplectic basis into $\C$ using the three embeddings contained in $\Phi$. The period matrix $Z$ is then $\Omega_2^{-1}\Omega_1$.

\section{Mumford's choice of symplectic basis and labeling of the branch points}\label{pathpicture}

We give here a representation of the symplectic basis and labeling of the branch points used by Mumford to carry out his computations. The image is  a projection of the paths to $\mathbb{P}^1(\C)$ under the map $(x,y) \mapsto x$. The branch points are labeled $a_1, a_2, \ldots a_{2g+1}, \infty$, the paths giving the symplectic basis for homology are labeled $A_i$, $B_i$, and the $\gamma_i$'s are the branch cuts of the projection. This image was created by R. Cosset for his PhD thesis \cite{Cossette}.

\begin{figure}
\begin{center}
\includegraphics[height=3in]{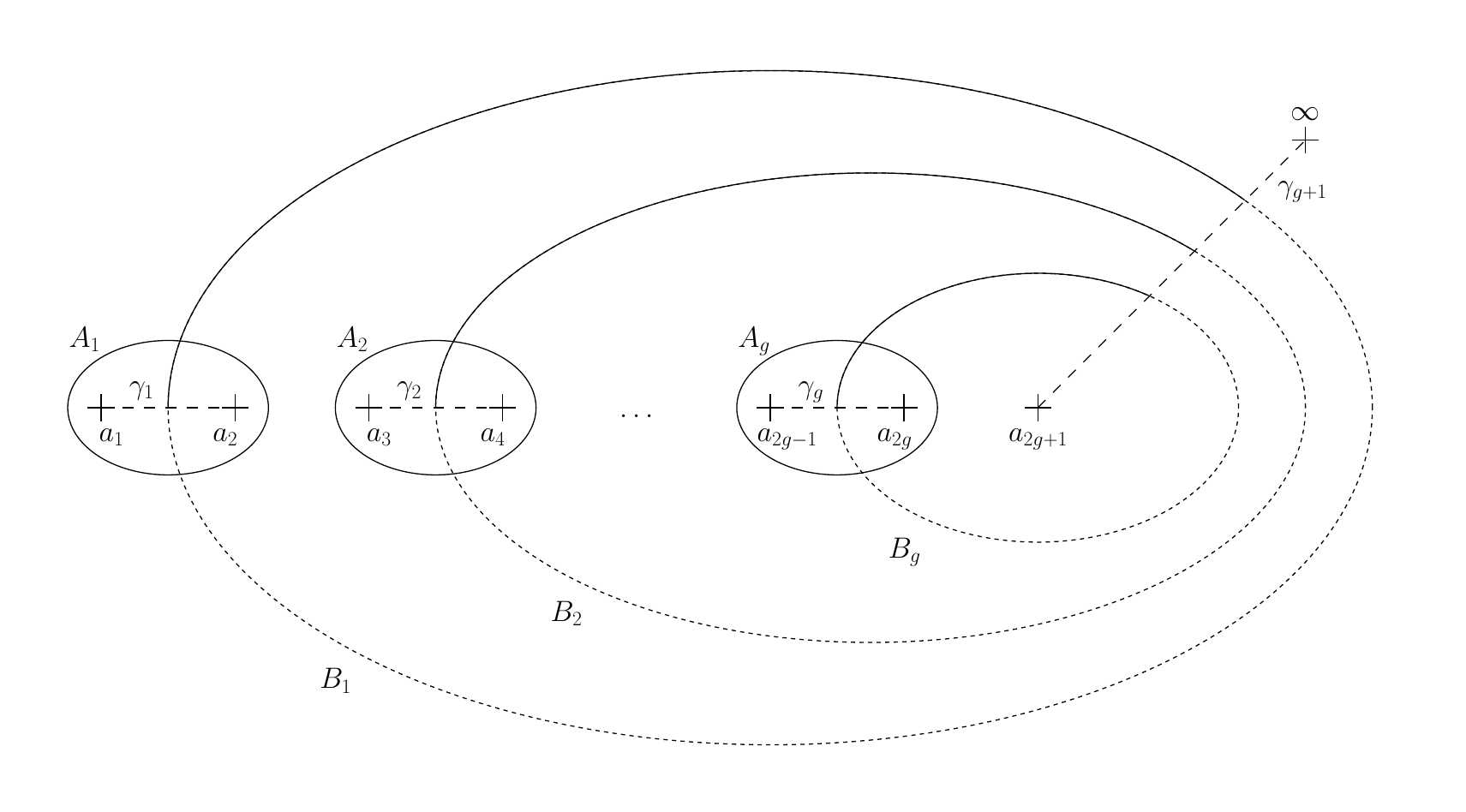}
\end{center}
\end{figure}
\pagebreak

\bibliography{biblio}{}
\bibliographystyle{amsplain}

\end{document}